\documentclass[a4paper,11pt]{article}
\usepackage[ansinew]{inputenc}
\usepackage[english]{babel}
\usepackage{amssymb}
\usepackage{amsmath}
\usepackage{amsthm}
\usepackage[pdftex]{color, graphicx}
\usepackage{mathrsfs}
\usepackage{verbatim}
\usepackage{stmaryrd}

\theoremstyle{definition} \newtheorem{defi}{Definition}[section]

\theoremstyle{plain} \newtheorem{thm}[defi]{Theorem}

\theoremstyle{plain} \newtheorem{prop}[defi]{Proposition}

\theoremstyle{plain} 

\theoremstyle{plain} 

\theoremstyle{definition} 

\theoremstyle{plain} \newtheorem{conj}[defi]{Conjecture}

\newcommand{\lb}{\llbracket}
\newcommand{\rb}{\rrbracket}
\newcommand{\soc}{\operatorname{soc}}

\newcommand{\ind}{\operatorname{ind}}

\newcommand{\ZX}{\mathbb{Z}[X]^W}

\author{Tobias Kildetoft}
\title{A Conjectural Algorithm for Simple Characters of Algebraic Groups}
\begin{document}

\maketitle


\abstract

We describe an algorithm, which - given the characters of tilting modules and assuming that Donkin's tilting conjecture is true - computes the characters of simple modules for an algebraic group in any characteristic.

\section{Introduction}

Let $G$ be a simple, simply connected algebraic group over an algebraically closed field $k$ of characteristic $p > 0$.

In this paper, we will describe an algorithm that computes the characters of the simple $G$-modules, given that we know the characters of the tilting modules and further assuming that Donkin's tilting conjecture is true.

Let $T\leq G$ be a maximal torus with associated root system $R$, simple roots $S$ and positive roots $R^+$. Let $X_+$ be the set of dominant weights and $B$ be the Borel subgroup corresponding to the negative roots.

For $\lambda\in X_+$ we get a $1$-dimensional $B$-module also denoted $\lambda$. Define $\nabla(\lambda) = \operatorname{ind}_B^G(\lambda)$, which is called the costandard module with highest weight $\lambda$. These have known characters, given by Weyl's character formula. The socle of $\nabla(\lambda)$ is simple and will be denoted $L(\lambda)$. These provide a non-redundant list of all simple $G$-modules.

A $G$-module is said to have a good filtration if it has a filtration with factors of the form $\nabla(\lambda)$ for suitable $\lambda$, and it is said to be tilting if both the module and its dual have good filtrations.

For each $\lambda\in X_+$ there is a unique indecomposable tilting module, all of whose weights $\mu$ satisfy $\mu\leq \lambda$ and which has $\nabla(\lambda)$ as the top factor in its good filtration. This will be denoted by $T(\lambda)$.

The algorithm will proceed in several steps. For the sake of clarity, we first describe the last part.

The final step will compute the characters of simple modules from the composition multiplicities in costandard modules. This step is not new, but a well-known idea.

Given the final step, the goal of the algorithm is to compute the composition multiplicities in costandard modules. This is done in two steps.

First, the character of the costandard module is written in what we call the $(p,r)$-basis. The coefficients arising from this are then calculated in terms of the multiplicities of costandard modules in tilting modules, under the assumption that Donkin's tilting conjecture is true. The calculation is recursive and assumes that all of these coefficients have been computed for smaller weights.

Next, the composition multiplicities in the modules from the $(p,r)$-basis are calculated. This calculation is again recursive and requires knowing composition multiplicities in costandard modules with (much) smaller weights.

Putting the two steps above together gives the composition multiplicities in costandard modules.

The algorithm involves choosing a natural number $r$. The choice of this $r$ will be dealt with later, but it is worth noting that if it is chosen large enough relative to the modules considered, most of the algorithm ``collapses'' to a single step, which is the same as one would do when $p\geq 2h-2$ where $h$ is the Coxeter number of the group. However, the tilting modules one needs to know the characters of become considerably more complicated as $r$ increases.

\section{Notation and Preliminaries}

From now on, we fix the following notation. We refer to \cite{rags} for further details on algebraic groups and their representations. For convenience, from now on, by ``$G$-module'' we will mean a finite dimensional and rational $G$-module.

\begin{itemize}
\item $k$ is an algebraically closed field of characteristic $p>0$.
\item $G$ is a semisimple, connected, simply connected algebraic group scheme over $k$, defined over $\mathbb{F}_p$.
\item $T\leq G$ is a maximal split torus.
\item $X = X(T)$ is the group of characters of $T$.
\item $R$ is the associated root system.
\item $S$ is a fixed basis of $R$.
\item $R^+$ is the set of positive roots corresponding to $S$.
\item $\alpha^{\vee}$ is the coroot associated to $\alpha\in R$.
\item $\langle\beta,\alpha^{\vee}\rangle$ is the natural pairing normalized such that $\langle\alpha,\alpha^{\vee}\rangle = 2$ for all $\alpha\in S$.
\item $\alpha_0$ is the highest short root of $R^+$.
\item $\rho = \frac{1}{2}\sum_{\alpha\in R^+}\alpha$ is the Weyl weight.
\item $h = \langle\rho,\alpha_0^{\vee}\rangle + 1$ is the Coxeter number of $R$.
\item $X_+ = \{\lambda\in X\mid \langle\lambda,\alpha^{\vee}\rangle \geq 0\mbox{ for all }\alpha\in R^+\}$ is the set of dominant weights.
\item $X_r = \{\lambda\in X_+\mid \langle\lambda,\alpha^{\vee}\rangle < p^r\mbox{ for all }\alpha\in S\}$ is the set of $r$-restricted weights for some integer $r\geq 1$.
\item $\leq$ is the partial order on $X$ defined by $\lambda\leq \mu$ iff $\mu - \lambda$ is a non-negative integral linear combination of positive roots.
\item $B\leq G$ is the Borel subgroup containing $T$ corresponding to the negative roots.
\item $W$ is the Weyl group of $R$.
\item $w_0\in W$ is the longest element.
\item $\lambda^* = -w_0(\lambda)$ is the dual weight of a weight $\lambda\in X$.
\item $\nabla(\lambda) = \ind_B^G(\lambda)$ is the costandard module of highest weight $\lambda$ for $\lambda\in X_+$.
\item $L(\lambda) = \soc_G\nabla(\lambda)$ is the simple module with highest weight $\lambda\in X_+$.
\item $T(\lambda)$ is the indecomposable tilting module with highest weight $\lambda$.
\item $M_{\lambda} = \{m\in M\mid t.m = \lambda(t)m\mbox{ for all }t\in T\}$ is the $\lambda$-weight space of the $G$-module $M$ for $\lambda\in X$.
\item $\mathbb{Z}[X]$ is the integral group ring of $X$ with basis $e(\lambda)$, $\lambda\in X$ such that $e(\lambda)e(\mu) = e(\lambda+\mu)$.
\item $\ZX$ is the set of $W$-fixed points of $\mathbb{Z}[X]$.
\item $[M] = \sum_{\lambda\in X}\dim(M_{\lambda})e(\lambda)\in \ZX$ is the character of the $G$-module $M$.
\item $F: G\to G$ is the Frobenius morphism which arises from the map $k\to k$ given by $x\mapsto x^p$.
\item $M^{(r)}$ is the $G$-module which as an additive group is the same as the $G$-module $M$, but with $G$-action composed with $F^r$.
\item $[M:L(\lambda)]_G$ is the composition multiplicity of the simple $G$-module $L(\lambda)$ in the $G$-module $M$.
\item $[M:\nabla(\lambda)]_{\nabla}$ is the multiplicity of $\nabla(\lambda)$ in a good filtration of the $G$-module $M$.
\item $\lb M,N\rb$ is the Euler characteristic of the modules $M$ and $N$, see \cite{kildetoft17} for more details as well as an alternative definition.
\end{itemize}

The following theorem is known as Steinberg's tensor product theorem. It will be used extensively in this paper, and will therefore just be referred to as such.

\begin{thm}[{\cite[Theorem 1.1]{steinberg63}},{\cite[Proposition II.3.16]{rags}}]
Let $\lambda\in X_+$ and write $\lambda = \lambda_0 + p^r\lambda_1$ with $\lambda_0\in X_r$.

Then $L(\lambda) \cong L(\lambda_0)\otimes L(\lambda_1)^{(r)}$.
\end{thm}

We will make extensive use of the form denoted by $\lb\cdot,\cdot\rb$ in \cite{kildetoft17}. Apart from the properties of this form we discuss as Conjecture \ref{conjecture}, the main properties we will need are:

\begin{prop}[{\cite[Proposition 5.1.1(3)]{kildetoft17}}]\label{formproperty}
If $M$ has a good filtration and $\lambda\in X_+$ then $\lb M,\nabla(\lambda)\rb = [M:\nabla(\lambda)]_{\nabla}$.
\end{prop}

\section{A Basis for the Characters}

For $\lambda = \lambda_0 + p^r\lambda_1\in X_+$ with $\lambda_0\in X_r$ we define $\nabla^{(p,r)}(\lambda) = L(\lambda_0)\otimes\nabla(\lambda_1)^{(r)}$.

\begin{prop}\label{prbasis}
The set $\{[\nabla^{(p,r)}(\lambda)]\mid\lambda\in X_+\}$ is a $\mathbb{Z}$-basis for $\mathbb{Z}X^W$.
\end{prop}

\begin{proof}
By \cite[Lemma 5.8]{rags} if $\chi\in \mathbb{Z}X^W$ we can write $\chi = \sum_{\lambda\in X_+}a_{\lambda}[L(\lambda)]$ for suitable $a_{\lambda}\in\mathbb{Z}$. By Steinbergs tensor product theorem, this becomes $$\sum_{\lambda_0\in X_r}\sum_{\lambda_1\in X_+}a_{\lambda_0,\lambda_1}[L(\lambda_0)][L(\lambda_1)^{(r)}]$$ with $a_{\lambda_0,\lambda_1} = a_{\lambda_0+p^r\lambda_1}$.

Further, by \cite[Remark 5.8]{rags} we can write $[L(\lambda_1)^{(r)}] = \sum_{\mu\in X_+}b_{\lambda_1,\mu}[\nabla(\mu)^{(r)}]$ for suitable $b_{\lambda_1,\mu}\in\mathbb{Z}$. Combining these we get $$\chi = \sum_{\lambda_0\in X_r}\sum_{\lambda_1\in X_+}\sum_{\mu\in X_+}a_{\lambda_0,\lambda_1}b_{\lambda_1,\mu}[L(\lambda_0)][\nabla(\mu)^{(r)}]$$ $$= \sum_{\lambda_0\in X_r}\sum_{\lambda_1\in X_+}\sum_{\mu\in X_+}a_{\lambda_0,\lambda_1}b_{\lambda_1,\mu}[L(\lambda_0)\otimes\nabla(\mu)^{(r)}]$$ $$= \sum_{\lambda_0\in X_r}\sum_{\lambda_1\in X_+}\sum_{\mu\in X_+}a_{\lambda_0,\lambda_1}b_{\lambda_1,\mu}[\nabla^{(p,r)}(\lambda_0 + p^r\mu)]$$ so the $\mathbb{Z}$-span of the set is all of $\mathbb{Z}X^W$.

On the other hand, the set is clearly linearly independent since each $\nabla^{(p,r)}(\lambda)$ has $\lambda$ as its unique highest weight.
\end{proof}

By the above, for any $G$-module $M$ we can define integers $a_M^r(\lambda)$ by writing uniquely $[M] = \sum_{\lambda\in X_+}a_M^r(\lambda)[\nabla^{(p,r)}(\lambda)]$. In particular, we will define $a_{\mu}^r(\lambda) = a_{\nabla(\mu)}^r(\lambda)$.

We will refer to the basis $\{[\nabla^{(p,r)}(\lambda)]\mid\lambda\in X_+\}$ as the $(p,r)$-basis.

\section{Donkin's Tilting Conjecture}

The validity of the algorithm presented in this paper is contingent on the following conjecture. By \cite[Corollary 6.3.2]{kildetoft17} this is equivalent to assuming that Donkin's tilting conjecture holds.

\begin{conj}\label{conjecture}
For all $\lambda,\nu\in X_r$ and all $\sigma,\mu\in X_+$ $$\lb T(2(p^r-1)\rho - \lambda^* + p^r\sigma),L(\nu)\otimes \nabla(\mu)^{(r)}\rb = \begin{cases}[T(\sigma):\nabla(\mu)]_{\nabla} & \mbox{if } \nu = \lambda \\ 0 & \mbox{else}\end{cases}$$
\end{conj}

From now on, we will assume that Conjecture \ref{conjecture} is true. Note that this is known to be the case when $p\geq 2h-2$ by \cite{donkin93}.

\section{Computing Coefficients in the $(p,r)$-basis}

For $\lambda,\mu\in X_+$ the number $a_{\mu}^r(\lambda)$ can be calculated recursively using the following. To apply it, we need to assume that we have calculated all $a_{\mu}^r(\nu)$ for $\nu < \lambda$.

\begin{prop}\label{prcoeff}
Let $\lambda,\mu\in X_+$ where $\lambda = \lambda_0 + p^r\lambda_1$ with $\lambda_0\in X_r$. Then $$a_{\mu}^r(\lambda) = [T(2(p^r-1)\rho - \lambda_0^*+p^r\lambda_1):\nabla(\mu)]_{\nabla} - \sum_{\sigma\in X_+,\,\sigma < \lambda_1}a_{\mu}^r(\lambda_0 + p^r\sigma)[T(\lambda_1):\nabla(\sigma)]_{\nabla}$$ In particular, if $\lambda_1$ is minimal in $X_+$ then $$a_{\mu}^r(\lambda) = [T(2(p^r-1)\rho - \lambda_0^*+p^r\lambda_1):\nabla(\mu)]_{\nabla}$$
\end{prop}

\begin{proof}
Since $\lb \cdot,\cdot\rb$ only depends on the characters of the modules involved, we see that by Conjecture \ref{conjecture} we have $$\lb T(2(p^r-1)\rho - \lambda_0^*+p^r\lambda_1),\nabla(\mu)\rb = \sum_{\sigma\leq \lambda_1}a_{\mu}^r(\lambda_0 + p^r\sigma)[T(\lambda_1):\nabla(\sigma)]_{\nabla}$$

Further, the left hand side becomes $[T(2(p^r-1)\rho - \lambda_0^*+p^r\lambda_1):\nabla(\mu)]_{\nabla}$ by Proposition \ref{formproperty} and since $[T(\lambda_1):\nabla(\lambda_1)]_{\nabla} = 1$ the claim follows.
\end{proof}

We thus see that if we know $[T(2(p^r-1)\rho - \lambda_0^*+p^r\lambda_1):\nabla(\mu)]_{\nabla}$ for all $\lambda_0\in X_r$ and all $\lambda_1,\mu\in X_+$ then we can determine $a_{\nu}(\sigma)$ for all $\nu,\sigma\in X_+$, assuming Conjecture \ref{conjecture}.

It is interesting to note that if we fix $\lambda$ and $\mu$, then we can pick a natural number $m$ such that $\lambda,\mu\in X_m$ and then for all $r\geq m$ the multiplicities $[T(2(p^r-1)\rho - \lambda^*):\nabla(\mu)]_{\nabla}$ become the same, as they are just the composition multiplicity of $L(\lambda)$ in $\nabla(\mu)$.

\section{Composition Numbers for Costandard Modules}

Once we know $a_{\nu}(\sigma)$ for all $\nu,\sigma\in X_+$ we can determine $[\nabla(\lambda):L(\mu)]_G$ for all $\lambda,\mu\in X_+$ by noting that $$[\nabla(\lambda):L(\mu)]_G = \sum_{\sigma\in X_+,\, \sigma\leq \lambda}a_{\lambda}(\sigma)[\nabla^{(p,r)}(\sigma):L(\mu)]_G$$ so we need to be able to calculate $[\nabla^{(p,r)}(\sigma):L(\mu)]_G$ for all $\sigma,\mu\in X_+$.

This can be done recursively using the following.

\begin{prop}\label{compmultinpr}
If $\mu = \mu_0 + p^r\mu_1$ and $\sigma = \sigma_0 + p^r\sigma_1$ with $\mu_0,\sigma_0\in X_r$ and $\mu_1,\sigma_1\in X_+$ then $$[\nabla^{(p,r)}(\sigma):L(\mu)]_G = \begin{cases}[\nabla(\sigma_1):L(\mu_1)]_G & \mbox{if }\sigma_0 = \mu_0 \\ 0 & \mbox{else}\end{cases}$$
\end{prop}

\begin{proof}
By Steinberg's tensor product theorem, if we take a composition series for $\nabla(\sigma_1)^{(r)}$ and tensor this with $L(\sigma_0)$ then we get a composition series for $L(\sigma_0)\otimes\nabla(\sigma_1)^{(r)} = \nabla^{(p,r)}(\sigma)$.

Since a composition series for $\nabla(\sigma_1)^{(r)}$ is obtained by applying $M\mapsto M^{(r)}$ to one for $\nabla(\sigma_1)$ the result follows.
\end{proof}

Note that in order to compute $[\nabla(\lambda):L(\mu)]_G$ we thus need to know $[\nabla(\sigma_1):L(\mu_1)]_G$ for all those $\sigma_1,\mu_1$ that can occur when we write $\sigma = \sigma_0 + p^r\sigma_1$ and $\mu = \mu_0 + p^r\mu_1$ with $\sigma_0,\mu_0\in X_r$ in the above sum. But these weights are smaller than both $\lambda$ and $\mu$ (in fact, much smaller, as we have divided by $p^r$), so we can repeat this until the weights occurring are small enough that all the costandard modules involved are simple, in which case the calculations can be done.

\section{Characters of Simple Modules}

Once we have the composition multiplicities in the costandard modules, we can obtain the characters of the simple modules using that we know the characters of the costandard modules and the fact that the change-of-basis matrix between the bases given by the characters of these modules is upper triangular unipotent. More precisely, we use that $$[L(\lambda)] = [\nabla(\lambda)] - \sum_{\mu < \lambda}[\nabla(\lambda):L(\mu)]_G[L(\mu)]$$ so once we know all the composition multiplicities in $\nabla(\lambda)$ and we have written the characters of all $L(\mu)$ with $\mu < \lambda$ as linear combinations of the characters of suitable $\nabla(\nu)$, then the above shows how to write the character of $L(\lambda)$ in the same form.

This means that in order to find the characters of all $L(\lambda)$ for $\lambda\in \Gamma$ for some $\Gamma\subseteq X_+$ we need to find the composition multiplicities of all simples in all $\nabla(\mu)$ for all $\mu$ such that $\mu\leq\lambda$ for some $\lambda\in \Gamma$, so we may as well pick a set $\Gamma$ which is saturated to begin with.

Further, by Stenberg's tensor product theorem, it suffices to find the characters of all $L(\lambda)$ for $\lambda\in X_1$, so we need a saturated set containing $X_1$. A natural choice for such a set is the set $\Gamma_1 = \{\lambda\in X_+\mid \langle\lambda,\alpha_0^{\vee}\rangle \leq (p-1)(h-1)\}$.

However, it may in some cases be useful to consider larger sets, as the effect of Steinberg's tensor product theorem is not obvious when the characters are written as linear combinations of the characters of the costandard modules.

\section{Choosing a suitable $r$}

Given $\lambda\in X_+$ we can choose $r$ large enough such that $\lambda\in X_r$, and if we do this, then $\nabla^{(p,r)}(\lambda) = L(\lambda)$ and if we apply Proposition \ref{prcoeff} we only get a single term, namely the good filtration multiplicity in the given tilting module.

As such, if we really do know the characters of all tilting modules equally well, regardless of how large the highest weights are, we can simply choose $r$ to be sufficiently large that all the weights we care about become $r$-restricted, in which case the algorithm becomes very simple.

For example, if one wants to consider all weights in $\Gamma_1$ as in the previous section, one will need to pick an $r$ such that $(p-1)(h-1) < p^r$.

However, the current description of the characters of tilting modules, as given in \cite{amrw17}, is highly recursive in nature if one wants to do computations. It may therefore be an advantage to limit the highest weights of the tilting modules one considers.

To limit this as much as possible, one can choose $r=1$, which transfers as much complexity as possible to this algorithm, rather than to the computations with tilting modules.

\section{Summary of the Algorithm}

In summary, to compute the character of $L(\mu)$ for a given $\mu\in X_+$ we do the following:

First, pick a suitable $r$, by picking $r$ as large as possible while still being able to efficiently compute the good filtration multiplicities in tilting modules of the form $T(2(p^r-1)\rho - \lambda_0^* + p^r\lambda_1)$ whenever $\lambda\leq\mu$ is written as $\lambda_0 + p^r\lambda_1$ with $\lambda_0\in X_r$ and $\lambda_1\in X_+$.

Next, find $a_{\mu}^r(\lambda)$ for all $\mu\leq\lambda$ using Proposition \ref{prcoeff} recursively.

For each $\lambda$ such that $a_{\mu}^r(\lambda)\neq 0$ compute the composition multiplicities $[\nabla^{(p,r)}(\lambda):L(\nu)]_G$ for all $\nu\leq \lambda$ using Proposition \ref{compmultinpr}. As part of this, one will need to also compute various composition multiplicities in costandard modules with smaller highest weights, which is done by repeating the algorithm starting with those weights, since in the next step, these are found for $\nabla(\mu)$.

Putting the above two steps together, one has found the composition multiplicities in $\nabla(\mu)$, allowing one to write $[L(\mu)] = [\nabla(\mu)] - \sum_{\lambda < \mu}b_{\mu}(\lambda)[L(\lambda)]$ for known constants $b_{\mu}(\lambda)$. Repeating the above for all the weights $\lambda$ with $b_{\mu}(\lambda)\neq 0$ will, once all the weights become small enough that that simple modules occurring are costandard modules, give $[L(\mu)]$ as a linear combination of costandard modules with known coefficients.

\bibliography{bibtex-full}
\bibliographystyle{alpha}


\end{document}